\documentclass[12pt]{amsart}

\usepackage[T2A]{fontenc}
\usepackage[cp1251]{inputenc}
\usepackage[english]{babel}
\usepackage{amsfonts,amsmath,amsxtra,amsthm,amssymb,latexsym}
\usepackage{srcltx}

\newtheorem{tht}{Theorem}
\newtheorem{thd}{Definition}
\newtheorem{thl}[tht]{Lemma}
\newtheorem{thp}[tht]{Proposition}
\newtheorem{thc}[tht]{Corollary}
 
\theoremstyle{definition}
\newtheorem{thex}{Example}
\newcommand{\cP}{{\mathcal{P}}}

\newcommand{\cS}{{\mathcal{S}}}

\newcommand{\cD}{{\mathcal{D}}}
\newcommand{\cJ}{{\mathcal{J}}}

\newcommand{\Hh}{{\mathcal{H}}}
\newcommand{\cG}{{\mathcal{G}}}

\newcommand{\cA}{{\mathcal{A}}}

\newcommand{\cM}{{\mathcal{M}}}

\newcommand{\Rep}{{\rm Rep}}

\newcommand{\dR}{\mathbb{R}}
\newcommand{\dN}{\mathbb{N}}
\newcommand{\dC}{\mathbb{C}}

\newcommand{\gt}{{\mathfrak{t}}}
\newcommand{\gA}{{\mathfrak{A}}}

\newcommand{\ii}{\mathrm{i}}

\newcommand{\ov}{\overline}

\author{Konrad Schm\"udgen}
\address{Universit\"at Leipzig, Mathematisches Institut, Augustusplatz 10/11, D-04109 Leipzig, Germany}
\email{schmuedgen@math.uni-leipzig.de}

\begin{title}
{Transition Probabilities of Positive Linear Functionals on $*$-Algebras}
\end{title}
\date{}
\begin{document}


\maketitle

\begin{abstract}
Using unbounded Hilbert space representations basic results on the transition probability of positive linear functionals $f$ and $g$ on a  unital $*$-algebra  are obtained. The main assumption is the essential self-adjointness of  GNS representations $\pi_f$ and $\pi_g$. Applications to functionals given by density matrices and  by integrals  and to vector functionals on the Weyl algebra are given. 
\end{abstract}


\textbf{Key  words:} transition probability, non-commutative probability, unbounded representations

\section {Introduction}
Let $f$ and $g$ be  
states on a unital $*$-algebra $A$. Suppose that these states are realized as vectors states  of  a common $*$-representation $\pi$ of $A$ on a Hilbert space  with unit vectors $\varphi$ and $\psi$, respectively, that is, $f(a)=\langle \pi(a)\varphi,\varphi\rangle$ and $g(a)=\langle \pi(a)\psi,\psi\rangle$ for $a\in A$. In quantum physics the number $|\langle \varphi,\psi\rangle|^2$ is then interpreted as the transition probability from $f$ to $g$  in these vector states. The (abstract) {\it transition probability} $P_A(f,g)$ is defined as the supremum of values $|\langle \varphi,\psi\rangle|^2$, where the supremum is taken over all realizations of $f$ and $g$ as vectors states in some common $*$-representation  of $A$. This definition was introduced  by A. Uhlmann \cite{uhlmann76}. The square root $\sqrt{P_A(f,g)}$ is often called  fidelity in the literature  \cite{alberti03},  \cite{josza}. 

The transition probability is  related to  other important topics such as the Bures distance \cite{bures}, Sakai's non-commutative Radon-Nikodym theorem \cite{araki72} and the geometric mean of Pusz and Woronowicz \cite{pw}. It was extensively studied in the finite dimensional case (see e.g. the monograph \cite{bengtson06}) and a number of results have been derived for $C^*$-algebras and von Neumann algebras (see e.g. \cite{araki72}, \cite{arraggio82}, \cite{alberti83},\cite{alberti92}, \cite{albuhl00}, \cite{alberti03}, \cite{yama}).

The aim of the  present paper is to study the transition  probability $P_A(f,g)$ for positive linear functionals $f$ and $g$ on a general unital $*$-algebra $A$. Then, in contrast to the case of $C^*$-algebras, the corresponding $*$-representations of $A$ act by unbounded operators in general and a number of   technical problems of   
 {\it unbounded} representation theory 
 on Hilbert space  come up. Dealing with these difficulties in a proper way is a main purpose  of this paper. In  section \ref{unbrep} we therefore  collect all basic definitions and facts on unbounded Hilbert space representations that will be used throughout this paper.

In section \ref{mainres} we state and prove our main  theorems about the transition  probability $P_A(f,g)$ for a general $*$-algebras. The crucial assumption for  these results is the essential self-adjointness of the GNS representations $\pi_f$ and $\pi_g$. This means that we restrict ourselves   to a class of "nice" functionals. In contrast we do not restrict the $*$-representation $\pi$ where the functionals $f$ and $g$ are realized as vector functionals. 
(In some results it is assumed that $\pi$ is closed or biclosed, but this is no restiction of generality, since any $*$-representation  has a closed or biclosed extension.) 

In section \ref{appl} we apply  Theorem \ref{sumofgnspifg1} from section \ref{mainres} to  generalize two standard formulas (\ref{fsfttrans0}) and (\ref{commutativecase}) for the transition  probability to the unbounded case; these  formulas concern trace functionals $f(\cdot)={\rm Tr}\, \rho(\cdot) t$\, and functionals  on commutative $*$-algebras given by integrals. A simple counter-example based on the Hamburger moment problem shows that these  formulas can fail if the assumption of essential self-adjointness of  GNS representations is omitted. In section \ref{vectorweyl} we determine the transition  probability of  positive functionals on the Weyl algebra given by certain functions from $C_0^\infty(\dR)$. In this case both GNS representations $\pi_f$ and $\pi_g$ are {\it not} essentially self-adjoint and the corresponding formula for $P_A(f,g)$ is in general  different from the standard formula (\ref{fsfttrans0}).

Throughout this paper we suppose that $A$ is a complex unital $*$-algebra. The involution of $A$ is denoted by $a\to a^+$ and the unit element of $A$  by $1$. Let $\cP(A)$ be the set of all positive linear functionals on $A$.  Recall that a linear functional $f$ on $A$ is called positive if $f(a^+a)\geq 0$ for all $a\in A$. Let $\sum A^2$ be the set of all finite sum of squares $a^+a$, where $a\in A$. All notions and  facts on von Neumann algebras and on unbounded operators used in this paper can be found in \cite{kr} and \cite{schmu12}, respectively.
\section{Basics on Unbounded Representations}\label{unbrep} 
Proofs  of all unproven facts stated in this section and more details can be found in the author's monograph \cite{schmu90}. Proposition \ref{adrepbasic2} below is a new result that might of interest in itself.

Let $(\cD,\langle\cdot,\cdot\rangle)$ be a unitary space and $(\Hh,\langle\cdot,\cdot\rangle)$  the Hilbert space  completion of $(\cD,\langle\cdot,\cdot\rangle)$. 
We denote by $L(\cD)$   the algebra of all linear operators\, $a:\cD\to \cD$, by $I_\cD$ the identity map of $\cD$ and by ${\bf B}(\Hh)$ the $*$-algebra of all bounded linear operators on $\Hh$. 
\begin{thd}
A\, {\em representation}\, of  $A$ on  $\cD$ is an algebra homomorphism $\pi$ of $A$ into the algebra $L(\cD)$ such that $ \pi(1)=I_\cD$ and $\pi(a)$ is a closable operator on $\Hh$ for  $a\in A$. We then write $\cD(\pi):=\cD$ and $\Hh(\pi):=\Hh$.

A\, {\em $*$-representation}\, $\pi$ of  $A$ on  $\cD$  is a representation $
\pi$ satisfying 
\begin{align}\label{defadjointre1}
\langle \pi(a)\varphi,\psi\rangle=\langle \varphi,\pi(a^+)\psi\rangle~~~~{\rm for}~~~a\in A,~\varphi, \psi\in \cD(\pi).
\end{align}
\end{thd}
Let\, $\pi$\, be a representation  of  $A$. Then
\begin{align}\label{defadjointre}
\cD(\pi^\ast):=\cap_{a\in A}\, \cD(\pi(a)^\ast)~~~{\rm and}~~~
 \pi^\ast(a):=\pi(a^+)^\ast\lceil\cD(\pi^\ast) ~~~{\rm for}~~~~a\in A,
\end{align}
 defines a representation $\pi^\ast$ of\,  $A$ on  $\cD(\pi^\ast)$, called the\, {\it adjoint representation}\, to $\pi$. Clearly, $\pi$ is a $*$-representation  if and only if\, $\pi\subseteq \pi^*$. 
 
If\, $\pi$ is a $\ast$-representation of  $A$, then
\begin{align}
 \cD(\, \ov{\pi}\, )&:=\cap_{a\in A}\, \cD(\, \ov{\pi(a)}\, )~~~{\rm and}~~~~ \ov{\pi}(a):=\ov{\pi(a)} \lceil\cD(\, \ov{\pi}\, ) ,~~a\in A,\label{defclosurere}\\
 \cD(\pi^{**})&:=\cap_{a\in A}\, \cD(\pi^*(a)^\ast)~~~{\rm and}~~~
 \pi^{**}(a):=\pi^*(a^+)^\ast\lceil\cD(\pi^{**}) ,~~a\in A,\label{defbiadjointre}
\end{align}
 are $\ast$-representations  $\pi^*$ and $\pi^{**}$\, of\, $A$,  called the\, {\it closure}\, resp. the\, {\it biclosure} of\, $\pi$. Then 
\begin{align*}
\pi\subseteq \ov{\pi}\subseteq \pi^{**}\subseteq \pi^\ast.
\end{align*}
If $\pi$ is  a $*$-representation, then  $\Hh(\pi)=\Hh(\pi^*)$. But for a representation $\pi$ it may happen that the domain $\cD(\pi^\ast)$ is {\it not} dense in $\Hh(\pi)$, that is,  $\Hh(\pi^*)\neq \Hh(\pi)$.
\begin{thp}\label{adrepbasic2}
Let\, $\pi$\, and\, $\rho$\, be representations of  a $*$-algebra $A$ such that\, $\rho\subseteq \pi$. Then:
\begin{itemize} 
\item[\rm (i)]~ $P_{\Hh(\rho)}\pi^*(a)\subseteq \rho^*(a)P_{\Hh(\rho)}$,  where $P_{\Hh(\rho)}$\, is the projection of\, $\Hh(\pi)$ onto $\Hh(\rho)$. 
\item[\rm (ii)]~ If\, $\Hh(\rho)=\Hh(\pi)$,\, then\, $\pi^*\subseteq \rho^*$. 
\item[\rm (iii)]~~ $\rho^{**}\subseteq \pi^{**}$.
\end{itemize}
\end{thp}

\begin{proof} 
(i): Let $P$ denote the projection  $P_{\Hh(\rho)}$ and fix $\psi\in \cD(\pi^*)$. Let $\varphi\in \cD(\rho)$ and $a\in A$. Using the assumption $\rho\subseteq \pi$  we obtain
\begin{align*}
\langle & \rho(a^+)\varphi ,P\psi\rangle =\langle P\rho(a^+)\varphi,\psi\rangle =\langle \rho(a^+)\varphi,\psi\rangle =\langle \pi(a^+)\varphi,\psi\rangle\\&= \langle \varphi,\pi(a^+)^*\psi\rangle=
\langle \varphi,\pi^*(a)\psi\rangle =\langle P\varphi,\pi^*(a)\psi\rangle=\langle 
\varphi,P\pi^*(a)\psi\rangle.
 \end{align*}
From this equality it follows that\, 
$P\psi \in \cD(\rho(a^+)^*)$\, and\, $\rho(a^+)^*P\psi=P\pi^*(a)\psi$. Hence\,
$\psi\in \cap_{b\in A}\, \cD(\rho(b)^*)=\cD(\rho^*)$\, and\, $\rho^*(a)P\psi=P\pi^*(a)\psi$. This proves that $P\pi^*(a)\subseteq \rho^*(a)P.$ 
 
 (ii) follows at once from (i), since $P=I$  by the assumption $\Hh(\rho)=\Hh(\pi)$.
 
 (iii):
 Let\, $\xi\in \cD(\rho^{**})$ and $\psi\in \cD(\pi^{*})$. Since\, $\Hh(\rho^{**})\subseteq \Hh(\rho^*)\subseteq \Hh(\rho)$\, by definition,  $P\xi=\xi$.  By (i),\, 
 $P\psi \in \cD(\rho^*)$\, and\, 
 $\rho^*(a)P\psi=P\pi^*(a)\psi$. Therefore, we derive
 \begin{align*}
 \langle & \pi^*(a)\psi ,\xi\rangle = 
 \langle  \pi^*(a)\psi ,P\xi\rangle =\langle P\pi^*(a)\psi,\xi\rangle\\& =\langle \rho^*(a)P\psi,\xi\rangle =\langle \psi,P\rho^{*}(a)^*\xi\rangle= \langle \psi,\rho^{**}(a^+)\xi\rangle
 \end{align*}
for  $a\in A$. Hence $\xi \in \cD(\pi^*(a)^*)$\,  and\, $\pi^*(a)^*\xi= \rho^{**}(a^+)\xi$\, for $a\in A$. This implies that\, $\xi\in\cD(\pi^{**})$\,  and\, $\pi^{**}(a^+)\xi=\pi^*(a)^*\xi= \rho^{**}(a^+)\xi$. Thus we have proved that\, $\rho^{**}\subseteq \pi^{**}$.
\end{proof}

\begin{thd} 
A $\ast$-representation\, $\pi$\, of a $\ast$-algebra $A$\,  is called\\ 
-- {\rm closed}\,  if $\pi=\ov{\pi}$,\, or equivalently, if\, $\cD(\pi)=\cD(\, \ov{\pi}\, )$,\\ 
-- {\rm biclosed}\,  if $\pi=\pi^{**}$,\, or equivalently, if\, $\cD(\pi)=\cD(\pi^{**} )$,\\ 
-- {\rm self-adjoint}\, if $\pi=\pi^\ast$,\, or equivalently, if\, $\cD(\pi)=\cD(\pi^\ast)$,\\
-- {\rm essentially self-adjoint}\, if\, $\pi^*$ is self-adjoint, that is, if\, $\pi^*=\pi^{**}$,\,  or equivalently, if\, $\cD(\pi^{**})=\cD(\pi^{*})$.
\end{thd}
Remark. 
It should be emphasized that the preceding definition of  {\it essential self-adjointness} is different  form the definition given in   \cite{schmu90}. In \cite[Definition 8.1.10]{schmu90},  a $*$-representation was called essentially self-adjoint if\, $\ov{\pi}$\, is self-adjoint, that is, if\, $\ov{\pi}=\pi^*$. 
\smallskip

Let $\pi$ be a $*$-representation. Then the $*$-representations $\ov{\pi}$ and $\pi^{**}$ are closed, $\pi^{**}$ is biclosed and $(\ov{\pi})^*=\pi^*$. It may  happen that $\ov{\pi}\neq \pi^{**}$, so  that $\ov{\pi}$ is closed, but not biclosed. The  locally convex topology on $\cD(\pi)$  defined by the family of seminorms $\{\|\cdot\|_a:=\|\pi(a)\cdot\|; a\in A\}$ is called the {\it graph topology} and denoted by $\gt_{\pi(A)}$. Then the $*$-representation $\pi$ is closed if and only if the locally convex space $\cD(\pi)[\gt_{\pi(A)}]$ is complete.

\begin{thp}\label{technicalself}
If\, $\pi_1$ is a self-adjoint  $\ast$-subrepresentation of a $*$-representation $\pi$ of $A$, then there exists a $*$-representation $\pi_2$ of $A$ on the Hilbert space $\Hh(\pi)\ominus \Hh(\pi_1)$ such that $\pi=\pi_1\oplus \pi_2$.
\end{thp}
\begin{proof} \cite[Corollary 8.3.3]{schmu90}. \end{proof}
For a $*$-representation of $A$ we define two {\it  commutants}
\begin{align*}
\pi(A)^\prime_s&=\{ T\in {\bf B}(\Hh(\pi)): T\varphi\in \cD(\pi),~~ T\pi(a)\varphi=\pi(a)T\varphi ~~~~ {\rm for} ~~~ a\in A, \varphi\in \cD(\pi)\},\\ \pi(A)^\prime_{ss}&=\{ T\in {\bf B}(\Hh(\pi)): T\, \ov{\pi(a)}\subseteq \ov{\pi(a)}\, T,~~T^*\, \ov{\pi(a)}\subseteq \ov{\pi(a)}\, T^*\, \}.
\end{align*}
The  symmetrized commutant\, $\pi(A)^\prime_{ss}$ is always a von Neumann algebra. If $\pi$ is closed, then  
\begin{align}\label{strongcom}
\pi(A)_{ss}^\prime=\pi(A)_s^\prime\cap(\pi(A)_s^\prime)^\ast.
\end{align}

If $\pi_1$ and $\pi_2$ are representations of $A$, the {\it interwining space} $I(\pi_1,\pi_2)$ consists of all bounded linear operators $T$ of $\Hh(\pi_1)$ into $\Hh(\pi_2)$ satisfying
\begin{align}\label{definterwiner}
T\varphi \in \cD(\pi_2)\quad {\rm and}\quad T\pi_1(a)\varphi=\pi_2(a)T\varphi \quad {\rm for}\quad a\in A, \varphi \in \cD(\pi_1).
\end{align}

The $*$-representation $\pi_f$  in the following proposition is called  the {\it GNS representation} associated with the positive linear functional $f$. 

\begin{thp}\label{gnsprop}
Suppose that $f\in \cP(A)$. Then there exists a $*$-representation $\pi_f$ with algebraically cyclic vector $\varphi_f$, that is, $\cD(\pi_f)=\pi_f(A)\varphi_f$, such that
\begin{align*}
f(a)=\langle \pi_f(a)\varphi_f,\varphi_f\rangle, ~~~ a\in A.
\end{align*}
If $ \pi$ is another $*$-representation of $A$ with algebraically cyclic vector $\varphi$ such that 
$f(a)=\langle \pi(a)\varphi,\varphi\rangle$ for all $ a\in A$, then there exists a unitary operator $U$ of $\Hh(\pi)$ onto $\Hh(\pi_f)$ such that 
$U\cD(\pi)=\cD(\pi_f)$ and
$\pi_f(a)=U^*\pi(a)U$ for $a\in A$.
\end{thp}
\begin{proof} \cite[Theorem 8.6.4]{schmu90}.
\end{proof}
We study some of  the preceding notions by a   simple but instructive example.
\begin{thex}\label{hambmp} ({\it One-dimensional Hamburger moment problem})\\
Let $A$ by the polynomial $*$-algebra $\dC[x]$ with involution determined by $x^+:=x$. We denote by $M(\dR)$ the set of positive Borel measures $\mu$ such that $p(x)\in L^1(\dR,\mu)$ for all $p\in \dC[x]$. 
The number $s_n=\int x^n d\mu(x)$ is the $n$-th moment  and the sequence $s(\mu)=(s_n)_{n\in \dN_0}$ is called the moment sequence of a measure $\mu\in M(\dR)$. The moment sequence $s(\mu)$, or likewise the measure $\mu$, is called {\it determinate}, if the moment sequence $s(\mu)$ determines the  measure $\mu$ uniquely, that is, if $s(\mu)=s(\nu)$ for some $\nu\in M(\dR)$ implies that $\nu=\mu$. 

For $\mu\in M(\dR)$ we define a $*$-representation $\pi_\mu$ of $A=\dC[x]$ by\, $\pi_\mu(p)q =p\cdot q$\, for $p\in A$ and $q\in \cD(\pi_\mu):=\dC[x]$ on the Hilbert space $\Hh(\pi_\mu):=L^2(\dR,\mu)$. Put $f_\mu(p)=\int p(x)d\mu(x)$ for $p\in\dC[x]$. Obviously, the vector $1\in \cD(\pi_\mu):=\dC[x]$ is algebraically cyclic for $\pi_\mu$. Therefore, since  $f_\mu(p)=\langle \pi_\mu(p)1,1\rangle$ for $p\in \dC[x]$, $\pi_\mu$ is (unitarily equivalent to) the GNS representation $\pi_{f_\mu}$ of the positive linear functional $f_\mu$ on $A=\dC[x]$.
\smallskip

\noindent{\bf Statement:}\, {\it  The $*$-representation $\pi_\mu$ is essentially self-adjoint if and only if the moment sequence $s(\mu)$ is determinate.}
\begin{proof}
By a well-known result on the Hamburger moment problem (see e.g. \cite[Theorem 16.11]{schmu12}), the moment sequence $s(\mu)$ is determinate if and only if the operator $\pi_\mu(x)$ is essentially sel-adjoint. By \cite[Proposition 8.1(v)]{schmu90}, the latter holds if and only if the $*$-representation $(\pi_\mu)^*$ is  self-adjoint, that is, if $\pi_\mu$ is essentially self-adjoint.
\end{proof} 

By \cite[Proposition 8.1(vii)]{schmu90}, the closure $\ov{\pi}_\mu$  of the $*$-representation $\pi_\mu$ is self-adjoint if and only if all powers of the operator $\pi_\mu(x)$ are essentially self-adjoint. 
This is a rather strong condition. It is fulfilled (for instance) if $1$ is an analytic vector for the symmetric operator $\pi_\mu(x)$, that is, if there exists a constant $M>0$ such that
$$\|\pi_\mu(x)^n 1\|= s_{2n}^{1/2}\leq M^n n!\quad {\rm for}\quad n\in \dN.$$ 
From the theory of moment problems it is well-known that there are examples of measures $\mu\in M(\dR)$ for which $\pi_\mu(x)$ is essentially self-adjoint, but $\pi_\mu(x^2)$ is not.  In this case $\pi_\mu$ is essentially self-adjoint (which means that $(\pi_\mu)^*$ is self-adjoint), but the closure\, $\ov{\pi}_\mu$ of\, $\pi_\mu$ is {\it not} self-adjoint.
\end{thex}

\section{Main Results on Transition Probabilities}\label{mainres}

Let $\Rep A$ denote the family of all $*$-representations  of $A$. Given $\pi\in \Rep A$ and  $f\in\cP( A)$, let $S(\pi,f)$ be the set of all representing vectors for the functional $f$ in $\cD(\pi)$, that is, $S(\pi,f)$ is the set of vectors $\varphi\in \cD(\pi)$ such that $f(a)=\langle \pi(a)\varphi,\varphi\rangle$ for\,  $a\in A$. Note that $S(\pi,f)$ may be empty, but by Proposition \ref{gnsprop} for each  $f\in \cP(A)$  there exists a $*$-representation $\pi$ of $A$ for which $S(\pi,f)$ is not empty. If $f$ is a state, that is, if $f(1)=1$, then all vectors $\varphi\in S(\pi,f)$ are unit vectors.

\begin{thd}
For $f,g \in\cP(A)$
the {\rm transition probability}\, $P_A(f,g)$\, of\, $f$ and $g$
is defined by
\begin{align}\label{transtionprob}
P_A(f,g)= \sup_{\pi \in \Rep A}~ \sup_{\varphi\in  \cS(\pi,f),\psi\in \cS(\pi,g)}~|\langle \varphi,\psi\rangle|^2.
\end{align}
\end{thd}
If $A$ is a unital $*$-subalgebra of $B$ and $f,g\in \cP(B)$, then it is obvious that
\begin{align}\label{pafgpbfg}
P_B(f,g)\leq \cP_A(f\lceil A,g\lceil A),
\end{align} because the restriction of any $*$-representation of $B$ is  a $*$-representation of $A$. 

Let 
$\cG(f,g)$ denote the set of all linear functionals on $A$ satisfying
\begin{align}\label{defigset}
 |F(b^+ a)|^2\leq f(a^+a)g(b^+b)\quad {\rm for}\quad a,b\in A.
 \end{align} 
 Any vector $\varphi\in S(\pi,f)$ is called an {\it amplitude} of $f$ in the representation $\pi$ and any linear functional of $\cG(f,g)$ is called a {\it transition form} from $f$ to $g$.
If $\varphi\in  \cS(\pi,f)$ and $\psi\in \cS(\pi,g)$, then the functional $F_{\varphi,\psi}$ defined by 
\begin{align}\label{transvarphipsi} F_{\varphi,\psi}(a):=\langle\pi(a)\varphi, \psi\rangle,\quad a\in A,
\end{align} is a transition form from $f$ to $g$. Indeed,  for\, $a,b\in A$ we have
\begin{align*}
|F_{\varphi,\psi}(b^+ a)|^2&=|\langle\pi(b^+a)\varphi, \psi\rangle|^2=|\langle \pi(a)\varphi,\pi(b)\psi\rangle|^2\\ &\leq \|\pi(a)\varphi \|^2\|\pi(b)\psi\|^2
= f(a^+a)g(b^+b)
\end{align*}
which proves that $F_{\varphi,\psi}\in \cG(f,g)$. By Theorem  \ref{chartransprob} below, {\it each} functional $F\in \cG(f,g)$ arises in this manner. 
The number\, $|F_{\varphi,\psi}(1)|^2=|\langle\varphi, \psi\rangle|^2$\, is called the {\it transition probability}  of the amplitudes $\varphi$ and $\psi$ and by definition the transition  probability\, $P_A(f,g)$\, is  the supremum of all such transition amplitudes.

The following  description of the transition probability  was proved by P.M. Alberti for $C^*$-algebras \cite{alberti83} and by A. Uhlmann for general $*$-algebras \cite{uhlmann85}. 
\begin{tht}\label{chartransprob}
Suppose that $f, \in \cP(A)$. 
Then
\begin{align}\label{cxhartranspobality}
P_A(f ,g)= \sup_{F\in \cG(f,g)}~|F(1)|^2.
\end{align}
There exist a $*$-representation $\pi$ of $A$ and vectors $\varphi \in S(\pi,f)$ and $\psi \in S(\pi,g)$
such that
\begin{align}\label{extemepsi12}
P_A(f ,g)=|\langle \varphi,\psi\rangle|^2.
\end{align}
\end{tht}
  Next we express the transition forms of $\cG(f,g)$ and hence the transition probability in terms of intertwiners of the corresponding GNS representations.  This provides a powerful tool for  computing the transition probability.
Recall $\pi_f$ denotes the GNS representation of $A$ associated with $f\in \cP(A)$ and $\varphi_f$ is the corresponding algebraically cyclic vector. 

\begin{thp}\label{proponetooneft}
Suppose that $f,g\in\cP(A)$. Then there a one-to-one correspondence between the sets\, $\cG(f,g)$\, and\, $I(\pi_f,(\pi_g)^*)$\, given by 
\begin{align}\label{onetoonegfd}
F(b^+a)= \langle T\pi_f(a)\varphi_f,\pi_g(b)\varphi_g\rangle\quad {\rm for} \quad a,b\in A,
\end{align}  where\, $F\in \cG(f,g)$\, and\, $T\in I(\pi_f,(\pi_g)^*)$.\, In particular,\, $F(1)=\langle T\varphi_f,\varphi_g\rangle$.
\end{thp}
\begin{proof}
Let $F\in \cG(f,g)$. Then 
$$
|F(b^+a)|^2\leq f(a^*a)g(b^*b)=\|\pi_f(a)\varphi_f\|^2\|\pi_g(b)\varphi_g\|^2 \quad {\rm for}\quad a,b\in A.
$$
Hence there exists a bounded linear operator $T$ of $\Hh(\ov{\pi}_g)$  into $\Hh(\ov{\pi}_f)$ such that\, $\|T\|\leq 1$\, and 
(\ref{onetoonegfd}) holds. Let $a,b,c\in A$. Using\, (\ref{onetoonegfd})\, we obtain
\begin{align*}
 \langle T \pi_f(a)\varphi_f,\pi_g(c^+)\pi_g(b)\varphi_g\rangle&=F((c^+b)^+a)=F(b^+(ca))\\&=\langle T\pi_f(c)\pi_f(a)\varphi_f,\pi_g(b)\varphi_g \rangle.
\end{align*}
Hence\, $T\pi_f(b)\varphi_f\in \cD(\pi_g(c)^*)$ and\, $\pi_g(c)^*T\pi_f(a)\varphi_f= T\pi_f(c^+)\pi_f(a)\varphi_f$.\, Because $c\in A$ was arbitrary,\, $T\pi_f(a)\varphi_f \in \cD((\pi_g)^*)$. Then $$(\pi_g)^*(c^+)T\pi_f(a)\varphi_f= T\pi_f(c^+)\pi_f(a)\varphi_f\quad {\rm for}\quad a\in A,$$ which means that $T\in I(\pi_f,(\pi_g)^*)$.

Conversely, let $T\in I(\pi_f,(\pi_g)^*)$ and $\|T\|\leq 1$. Define  $F(a)= \langle T\pi_f(a)\varphi_f,\varphi_g\rangle$\, for\, $a\in A$. It is straightforward to check that\,  (\ref{onetoonegfd})\, holds and hence\, (\ref{defigset}), that is, $F\in \cG(f,g)$. 

Clearly,  by\, (\ref{onetoonegfd}), $F=0$ is equivalent to $T=0$. Thus we have a one-to-one correspondence between functionals $F$ and operators $T$.
\end{proof}

Combining Theorem \ref{chartransprob} and Proposition \ref{proponetooneft} and using the formula $F(1)=\langle T\varphi_f,\varphi_g\rangle$ we obtain
\begin{thc}
For any $f, g \in \cP(A)$ we have  
\begin{align}\label{intertwtnull}
P_A(f,g)=\sup_{T\in I(\pi_f,(\pi_g)^*),\, \|T\|\leq 1}~ |\langle T \varphi_f,\varphi_g\rangle|^2.
\end{align}
\end{thc}
If the GNS representations of $f$ and $g$ are essentially self-adjoint, a number of stronger results can be obtained.    
\begin{tht}\label{sumofgnspifg}
Suppose that $f$ and $g$ are positive linear functionals on $A$ such that their GNS representations $\pi_f$ and $\pi_g$ are essentially self-adjoint. Let $\pi$ be a biclosed $\ast$-representation of $A$ such that the sets $S(\pi,f)$ and $S(\pi,g)$ are not empty. Fix vectors 
$\varphi \in S(\pi,f)$ and $\psi \in S(\pi,g)$.
Then
\begin{align}\label{proptrastcss}
P(f,g)=\sup_{T\in \pi(A)^\prime_{ss}, \|T\|\leq 1} ~|\langle T \varphi,\psi\rangle|^2.
\end{align}
\end{tht} 
\begin{proof} Let $T\in \pi(A)^\prime_{ss}$ and $ \|T\|\leq 1$. Similarly, as in the proof of Proposition  \ref{proponetooneft}, we define   
$F(a)= \langle T\pi(a)\varphi,\psi\rangle$, $a\in A$. Since $\pi(A)^\prime_{ss} \subseteq\pi(A)^\prime_{s}$, we obtain 
\begin{align*}
|F(b^+a)|^2&=|\langle T\pi(b^+a)\varphi,\psi\rangle|^2=|\langle \pi(b^+)T\pi(a)\varphi,\psi\rangle|^2\\&=|\langle T\pi(a)\varphi,\pi(b)\psi\rangle|^2\leq \|\pi (a)\varphi\|^2~\|\pi(b)\psi\|^2=f(a^+a)g(b^+b)
\end{align*}

for $a,b\in A$, that is, $F\in \cG(f,g)$. Clearly, we have $\langle T\varphi,\psi\rangle=F(1)$. 
Let $\rho_f$ and $\rho_g$ denote  the restrictions $\pi \lceil \pi(A)\varphi$ and $\pi \lceil \pi(A)\psi,$ respectively.  Since $\rho_f\subseteq \pi$ and $\rho_g\subseteq \pi$ and  $\pi$ is  biclosed, it follows from Proposition \ref{adrepbasic2}(iii) that $(\rho_f)^{**}\subseteq \pi^{**}=\pi$ and $(\rho_g)^{**}\subseteq \pi^{**}=\pi$. Since $\varphi \in S(\pi,f)$ and $\psi \in S(\pi,g)$, the representations $\rho_f$ and $\rho_g$ are unitarily equivalent to the GNS representations $\pi_f$ and $\pi_g$, respectively. For notational simplicity we identify $\rho_f$ with $\pi_f$ and $\rho_g$ with  $\pi_g$. Since $\rho_f$ and $\rho_g$ are essentially self-adjoint by assumption,  $(\rho_f)^{**}$ and $(\rho_g)^{**}$ are self-adjoint. Therefore, by Proposition \ref{technicalself}, there are subrepresentations $\rho_1$ and $\rho_2$ of $\pi$ such that $\pi=(\rho_f)^{**}\oplus \rho_1$ and $\pi=(\rho_g)^{**}\oplus \rho_2$.  

Conversely, suppose  that  $F\in \cG(f,g)$. By Proposition \ref{proponetooneft}, there is an intertwiner  $T_0\in I(\rho_f,(\rho_g)^*)\cong I(\pi_f,(\pi_g)^*))$\, such that\, $\|T_0\|\leq 1$\, and (\ref{onetoonegfd})\, holds with $T$ replaced by $T_0$. Define $T:\Hh(\rho_f) \oplus \Hh(\rho_1)\to \Hh(\rho_g) \oplus \Hh(\rho_2)$ by $T(\xi_f,\xi_1)=(T_0\xi_f,0)$. Clearly, $T^*$ acts by $T^*(\eta_g,\eta_2)=(T_0^*\eta_g,0)$. Since $(\rho_g)^{**}=(\rho_g)^*$ and $(\rho_f)^{**}=(\rho_f)^*$ by assumption and\, $T_0\in I(\rho_f,(\rho_g)^*)$, it follows from Proposition 8.2.3,(iii) and (iv), in \cite{schmu90} that 
\begin{align*} 
T_0&\in I((\rho_f)^{**},(\rho_g)^*)=I((\rho_f)^{**},(\rho_g)^{**}),\\ T_0^* &\in I((\rho_g)^{**},(\rho_f)^*)=I((\rho_g)^{**},(\rho_f)^{**}).
\end{align*}  
From these relations we easily derive  that the operators $T$ and $T^*$ are in $\pi(A)^\prime_{s}$,  so that  $T\in \pi(A)^\prime_{ss}$ by (\ref{strongcom}). Then we have $\|T\|=\|T_0\|\leq 1$ and 
$F(1)=\langle T_0 \varphi,\psi\rangle=\langle T \varphi,\psi \rangle$. Togther with the first  paragraph of this proof we have shown that the supremum over the operators $T\in \pi(A)^\prime_{ss}$, $ \|T\|\leq 1$, is equal to the supremum over the functionals $F\in \cG(f,g)$. Since the latter is equal to $P_A(f,g)$ by Theorem \ref{chartransprob}, this proves (\ref{proptrastcss}).
\end{proof}

Remark. A slight modification of the preceding proof shows the following: If we assume that the closures $\ov{\pi}_f$ and $\ov{\pi}_g$ of the GNS representations $\pi_f$ and $\pi_g$ are self-adjoint, then the assertion of Theorem \ref{sumofgnspifg}  
remains valid if it is only assumed that $\pi$ is {\it closed} rather than biclosed. A similar remark applies  also for the  subsequent applications of Theorem \ref{sumofgnspifg}  given below.
\smallskip

Theorem \ref{sumofgnspifg} says that (in the case of essentially self-adjoint GNS representations $\pi_f$ and $\pi_g$)  the transition probability $P(f,g)$ is given by  formula (\ref{proptrastcss}) in {\it any} fixed biclosed $*$-representation $\pi$ for which the sets $S(\pi,f)$ and $S(\pi,g)$ are not empty and for {\it arbitrary} fixed vectors $\varphi \in S(\pi,f)$ and $\psi \in S(\pi,g)$. In particular, we may take $\pi:=(\pi_f)^{**}\oplus (\pi_g)^{**}$,\, $\varphi:=\varphi_f$,\, and $\psi:=\varphi_g.$
  \begin{tht}\label{sumofgnspifg1}
Suppose that $f, g\in \cP(A)$ and the GNS representations $\pi_f$ and $\pi_g$ are essentially self-adjoint. Suppose that  $\pi$ is a biclosed $\ast$-representation of $A$ and there exist vectors
 $\varphi \in S(\pi,f)$ and $\psi \in S(\pi,g)$. Let $F_\varphi$ and $F_\psi$ denote the vector functionals on the von Neumann algebra\, $\cM:=(\pi(A)^\prime_{ss})^\prime$ given by $F_\varphi (x)=\langle x\varphi,\varphi\rangle $ and $F_\psi (x)=\langle x\psi,\psi\rangle $,\, $x\in \cM$.
 Then we have
 \begin{align}\label{proptrastcss1}
 P_A(f,g)=P_{\cM}\, (F_\varphi,F_\psi) .
 \end{align}
  Further, there exist vectors $\varphi^\prime\in S(\pi,f)$ and  $\psi^\prime\in S(\pi,g)$ such that 
  $\langle x \varphi^\prime,\varphi^\prime\rangle =\langle x \varphi,\varphi\rangle$ and $\langle x \psi^\prime,\psi^\prime\rangle =\langle x \psi,\psi\rangle$ for $x\in \cM$ and
 \begin{align}\label{attaiendtp}
 P_A(f,g)=|\langle\varphi^\prime,\psi^\prime\rangle|^2.
 \end{align}
 \end{tht}
\begin{proof} Since $\pi(A)^\prime_{ss}$ is a  von Neumann algebra, we have $T\in \pi(A)^\prime_{ss}$ if and only if\, $T\in (\pi(A)^\prime_{ss})^{\prime\prime}=\cM^\prime $ . Therefore, applying formula (\ref{proptrastcss}) to the $*$-representation $\pi$ of $A$ and to the identity representation of the von Neumann algebra $\cM$, it follows that the supremum of\,  $|\langle T \varphi,\psi\rangle|^2$\, over all operators\, $T\in \pi(A)^\prime_{ss}=\cM^\prime$,\, $ \|T\|\leq 1$, \, is equal to $P_A(f,g)$ and also to $P_{\cM}\, (F_\varphi,F_\psi)$. This yields the equality (\ref{proptrastcss1}).  

Now we prove the existence of vectors $\varphi^\prime$ and  $\psi^\prime$ having the desired properties. In order to do so we go into the details of the proof of \cite[Appendix 7]{alberti92}. Besides we  use some  facts from von Neumann algebra theory \cite{kr}. We define a normal linear functional on the von Neumann algebra $\cM^\prime$  by $h(\cdot)=\langle \cdot \, \varphi,\psi\rangle$. Let $h=R_u|h|$ be the polar decomposition of $h$, where $u$  is a partial isometry from $\cM^\prime$. Then we have $|h|=R_{u^*}h$ and hence $\|h\|=\|\, |h|\, \|= |h|(1)=h(u^*)=\langle u^*\varphi,\psi\rangle$. Therefore, we obtain
\begin{align}\label{hnorm}
P_{\cM}\, (F_\varphi,F_\psi)=\sup_{T\in \cM^\prime, \|T\|\leq 1} ~|\langle T \varphi,\psi\rangle|^2= \|h\|^2=\langle u^*\varphi,\psi\rangle^2,
\end{align}
where the first equality follows formula (\ref{proptrastcss}) applied  to the von Neumann algebra $\cM$. In the  proof of \cite[Appendix 7]{alberti92} it was shown that there exist partial isometries $v,w\in \cM^\prime $ satisfying 
\begin{align}\label{uvw} &\langle u^*\varphi,\psi\rangle=\langle v^*w\varphi,\psi\rangle,\\ &   w^*w\geq p(\varphi), \, v^*v\geq p(\psi),\label{vwvarphi}
\end{align} 
where $p(\varphi)$ and $p(\psi)$ are the projections of $\cM^\prime$ onto the closures of $\cM \varphi$ and $\cM \psi$, respectively. 
Set   $\varphi^\prime:=w \varphi$ and   $\psi^\prime:=v \psi$. Comparing (\ref{uvw}) with (\ref{hnorm}) and (\ref{proptrastcss1}) we obtain (\ref{attaiendtp}).

From (\ref{vwvarphi}) it follows that $\langle x \varphi^\prime,\varphi^\prime\rangle =\langle x \varphi,\varphi\rangle$ and $\langle x \psi^\prime,\psi^\prime\rangle =\langle x \psi,\psi\rangle$ for $x\in \cM$ and that $w^*w\varphi=\varphi$ and $v^*v\psi=\psi$. Since  $w,w^*\in \cM^\prime=\pi(a)^\prime_{ss}$ and $\pi$ is closed, we have $w,w^*\in \pi(a)^\prime_{s}$  by (\ref{strongcom}. Therefore,   $w$ and $w^*$ leave the domain $\cD(\pi)$ invariant, so that $\varphi^\prime=w \varphi\in \cD(\pi)$ and $\psi^\prime=v \psi\in \cD(\pi)$. For $a\in A$ we derive
\begin{align*}
\langle\pi(a)\varphi^\prime,\varphi^\prime\rangle&= \langle\pi(a)x\varphi,w\varphi\rangle= \langle w^* \pi(a)w\varphi,\varphi\rangle\\ &= \langle  \pi(a)w^*w\varphi,\varphi\rangle=\langle  \pi(a)\varphi,\varphi\rangle =f(a).
\end{align*}
That is, $\varphi^\prime\in S(\pi,f)$. Similarly,  $\psi^\prime\in S(\pi,g)$.
\end{proof}
Theorem \ref{proptrastcss1} allows us to reduce the computation of the transition probability  of the functionals $f$ and $g$ on $A$ to that of the vector functionals $F_\varphi$ and $F_\psi$ of the von Neumann algebra\, $\cM=(\pi(A)^\prime_{ss})^\prime$.
In the next section we will apply this result in two important situations.

The following theorem generalizes a  classical result of A. Uhlmann \cite{uhlmann76} to the unbounded case.
\begin{tht}
Let $f, g\in \cP(A)$ be such  that the GNS representations $\pi_f$ and $\pi_g$ are essentially self-adjoint. Suppose that there exist a positive linear  functional $h$ on $A$ and elements $b, c\in A$ such that  $f(a)=h(b^+ab)$ and $g(a)=h(c^+ac)$ for $a\in A$. 
Assume that\, $c^+b\in \sum A^2$. 
Then\, $$P_A(f,g)=h(c^+b)^2.$$
\end{tht}
\begin{proof}
Recall that $\pi_h$ is the GNS representation of $h$  with algebraically cyclic vector $\varphi_h$. By the assumptions $f(\cdot)=h(b^+\cdot b)$ and $g(\cdot)=h(c^+\cdot c)$ we  have 
$\pi_h(b)\varphi_h\in S(\pi_h,f)$ and $\pi_h(c)\varphi_h\in S(\pi_h,g)$. Therefore,  $$h(c^+b)^2=\langle \pi_h(b)\varphi_h,\pi_h(c)\varphi \rangle^2 \leq P_A(f,g).$$

To prove the converse inequality we want to apply Theorem \ref{sumofgnspifg} to the biclosed representation\,  $\pi:=(\pi_h)^{**}$. Suppose that\, $T\in \pi(A)^\prime_{ss}$ and $\|T\|\leq 1$.
Set $R:=\pi(c^+b)$. Since\, $c^+b \in\sum\, A^2$\, by assumption, $R$ is a positive, hence symmetric, operator. Since\, $\pi:=(\pi_h)^{**}$\, is closed, we  have\, $T\in  \pi(A)^\prime_{s}$.
Using these facts and the Cauchy-Schwarz inequality we derive 
\begin{align}\label{tvarphibc1}
|\langle T\pi_h(b)\varphi_h ,\pi_h(c)\varphi_h\rangle|^2 &=|\langle T\pi(b)\varphi_h,\pi(c)\varphi_h\rangle|^2
= |\langle \pi(b)T\varphi_h,\pi(c)\varphi_h\rangle|^2\nonumber\\&=| \langle RT\varphi_h,\varphi_h\rangle|^2 \leq \langle RT\varphi_h,T\varphi_h\rangle \langle R\varphi_h,\varphi_h\rangle\nonumber \\& =\langle RT\varphi_h,T\varphi_h\rangle h(c^+b). 
\end{align}
Since $T\in \pi(A)^\prime_{ss}$, we have $TR\subseteq RT$. There exists a positive self-adjoint extension $\tilde{R}$ of $R$ on $\Hh(\pi)$ such that\, $T\tilde{R}\subseteq \tilde{R}T$ \cite[Exercise 14.14]{schmu12}. The latter implies that\, $T\tilde{R}^{1/2}\subseteq \tilde{R}^{1/2}T$\, and hence
\begin{align}\label{tvarphibc2}
\langle RT\varphi_h,T\varphi_h\rangle &=\langle \tilde{R}T\varphi_h,T\varphi_h\rangle=\langle \tilde{R}^{1/2}T\varphi_h,\tilde{R}^{1/2}T\varphi_h\rangle
\nonumber \\ &=\langle T \tilde{R}^{1/2}\varphi_h,T\tilde{R}^{1/2}\varphi_h\rangle\nonumber \leq \langle \tilde{R}^{1/2}\varphi_h,\tilde{R}^{1/2}\varphi_h\rangle\nonumber \\&=
\langle \tilde{R} \varphi_h,\varphi_h\rangle=\langle R \varphi_h,\varphi_h\rangle=\langle \pi_h(c^+b) \varphi_h,\varphi_h\rangle=h(c^+b)
\end{align}
Inserting (\ref{tvarphibc2}) into (\ref{tvarphibc1}) we  get $$|\langle T\pi_h(b)\varphi_h,\pi_h(c)\varphi_h\rangle|^2\leq h(c^+b).$$ Hence $P_A(f,g)\leq h(c^+b)$\, by Theorem \ref{sumofgnspifg}.
\end{proof}
Remarks. 1. The assumption $c^+b\in \sum A^2$ was only needed to ensure that the operator\, $R=\pi(c^+b)\equiv(\pi_h)^{**}(c^+b)$\, is  {\it positive}. Clearly, this is satisfied if\, $F(c^+b)\geq 0$\, for all positive linear functionals $F$ on $A$. 

2. If the closures of the GNS representations $\pi_f$ and $\pi_g$ are self-adjoint, we can set\, $\pi:=\ov{\pi}_h$ in the  preceding proof and it suffices to assume that\, $h(a^+c^+ba)\geq 0$\, for all $a\in A$ instead of\, $c^+b\in \sum A^2.$
\section{Two Applications}\label{appl}

To formulate our first application we begin with  some  preliminaries. 

Let  $\rho$ be a closed $*$-representation of $A$. We denote by 
${\bf B}_1(\rho(A))_+$  the set of positive trace class operators on $\Hh(\rho)$ such that $t\Hh(\rho)\subseteq \cD(\rho)$ and the closure of  $\rho(a)t\rho(b)$ is  trace class   for all $a,b\in A$. 

Now let $t\in{\bf B}_1(\rho(A))_+$. We define a positive linear functional $f_t$ by
$$
f_t(a):={\rm Tr}\, \rho(a)t,\quad a\in A,
$$
where ${\rm Tr}$ always denotes the trace on the Hilbert space $\Hh(\rho)$. Note that $f_t(a)\geq 0$ if  $\rho(a)\geq 0$ (that is, $\langle \rho(a)\varphi,\varphi\rangle \geq 0$ for all $\varphi\in \cD(\rho)$).

In unbounded representation theory  a large class of positive linear functionals is  of the form $f_t$. We illustrate this by restating the  following theorem proved in \cite{schmu78}. Recall that a {\it Frechet--Montel space} is a complete metrizable locally convex space such that each bounded sequence has a convergent subsequence.  

\begin{tht}\label{tracereptheorem}
Let $f$ be a linear functional on $A$ and let $\rho$ be a closed $*$-representation of $A$. Suppose   that the locally convex space\, $\cD(\rho)[\gt_{\rho(A)}]$\, is a Frechet-Montel space and $f(a)\geq 0$ whenever  $\rho(a)\geq 0$  for $a\in A$. Then 
there  exists 
an operator\, $t\in {\bf B}(\rho(A))_+$\, such that\, $f=f_t$, that is, $f(a)={\rm Tr}~ \rho(a)t$\, for $a\in \cA$.
\end{tht}

Further, let $\cM$ be a type $I$ factor acting on  the Hilbert space $\Hh(\rho)$ and let ${\rm tr_\cM}$ denote its canonical trace. Since in particular $t$ is of trace class,   $F_t(x)={\rm Tr}\, x t$, $x\in \cM$, defines a positive normal linear functional $F_t$ on $\cM$. Hence there exists a unique positive element $\hat{t}\in \cM$ such that\, ${\rm tr}_\cM\, (\, \hat{t}\, )<\infty$ and 
\begin{align}\label{that}
F_t(x)\equiv {\rm Tr}\, x t={\rm tr}_\cM\, x\hat{t}\quad {\rm for}\quad x\in \cM.
\end{align}
The element $\hat{t}$ can be obtained as follows. Since $\cM$ is a type $I$ factor, there exist Hilbert spaces $\Hh_0$ and $\Hh_1$ such that, up to unitary equvalence,  $\Hh(\pi)=\Hh_0\otimes \Hh_1$\, and\, $\cM={\bf B}(\Hh_0)\otimes \dC\cdot I_{\Hh_1}$. The canonical trace of $\cM$ is then given by\, ${\rm tr}_\cM (y\otimes \lambda\cdot I_{\Hh_1}):= {\rm Tr}_1\,\lambda y$, where\, ${\rm Tr}_1$ denotes  the trace on the Hilbert space $\Hh_1$. Now $\tilde{F}_t(y):= F_t(y\otimes I_{\Hh_1})$, $y\in{\bf B}(\Hh_0)$,  defines a positive normal  linear functional $\tilde{F}_t$ on ${\bf B}(\Hh_0)$. Hence there exists a unique positive trace class operator $\tilde{t}$ on the Hilbert  space $\Hh_0$ such  $\tilde{F}_t(y)={\rm Tr}\, y\tilde{t}$\, for $y\in{\bf B}(\Hh_0)$. Set $\hat{t}:=\tilde{t}\otimes I_{\Hh_1}$. Then we have ${\rm tr}_\cM\, \hat{t}={\rm Tr}_1\, t< \infty$ and (\ref{that}) holds by construction. 

Note that ${\rm tr}_\cM= {\rm Tr}$\, and\, $t=\hat{t}$\, if\, $\cM={\bf B}(\Hh(\rho))$.

\begin{tht}\label{tracegen}
Let\, $\rho$\, be a closed $*$-representation of $A$ such that  the von Neumann algebra $\cM:=(\rho(A)^\prime_{ss})^\prime$ is  a type $I$ factor. 
For $s,t\in {\bf B}(\rho(A))_+$, let $f_s, f_t $ denote the positive linear  functionals on $\cA$ defined by
\begin{align*}
f_s(a)={\rm Tr}\, \rho(a)s, \quad f_t(a)={\rm Tr}\, \rho(a)t \quad {\rm for}\quad a\in \cA.
\end{align*} Suppose that the GNS representations  $\pi_{f_s}$ and $\pi_{f_t}$ are essentially self-adjoint. Then 
\begin{align}\label{fsfttrans0}
P_A(f_s,f_t)= ({\rm tr}_\cM\,|\hat{t}^{1/2}\hat{s}^{1/2}|)^2= (\,{\rm tr}_\cM\, (\hat{s}^{1/2}\, \hat{t}\, \hat{s}^{1/2})^{1/2}\, )^2.
\end{align}
\end{tht}
\begin{proof}
Let $\rho_\infty$ be the orthogonal sum 
$\oplus_{n=0}^\infty~ \rho$ on\, $\Hh_\infty=\oplus_{n=0}^\infty~ \Hh(\rho)$.\,  Since $\rho$ is  biclosed, so is the $*$-representation  $\rho_\infty$ of $A$.
We want to apply Theorem \ref{sumofgnspifg1}.
First we will describe the GNS representations\, $\pi_{f_s}$\, and\, $\pi_{f_t}$\, as  $*$-subrepresentations of\, $\rho_\infty$. 

The result is well-known if $\Hh(\rho)$ is finite dimensional \cite{uhlmann76}, so we can assume that $\Hh(\rho)$ is infinite dimensional.
Since $s\in{\bf B}_1(\rho(A))_+$,  there are a sequence $(\lambda_n)_{n\in \dN}$ of nonnegative numbers and an orthonormal sequence $(\varphi_n)_{n\in \dN}$ of $\Hh(\rho)$ such that $\varphi_n\in \cD(\rho)$ for $n\in \dN$, 
$$s\varphi =\sum_n \langle \varphi, \varphi_n\rangle\lambda_n\varphi_n \quad{\rm for}\quad \varphi\in \Hh(\rho),$$
and $(\rho(a)\lambda_n^{1/2}\varphi_n)_{n\in \dN}\in \Hh_\infty$ for all $a\in \cA$. Further, for $a\in A$ we  have 
\begin{align}\label{f_s}
 f_s(a)=\sum_{n=1}^\infty\, \langle \rho(a) \varphi_n,\lambda_n\varphi_n\rangle. 
 \end{align}
All these facts are contained in Propositions 5.1.9 and 5.1.12 in \cite{schmu90}. Hence  
$$\rho_\Phi(a)(\rho(b)\lambda_n^{1/2}\varphi_n):=(\rho(ab)\lambda_n^{1/2}\varphi_n),\quad a,b\in \cA,$$
 defines a $\ast$-representation $\rho_\Phi$ of $\cA$ on the domain $$\cD(\rho_\Phi):=\{(\rho(a)\lambda_n^{1/2}\varphi_n)_{n\in \dN};a\in \cA\}$$ with algebraically cyclic vector\,  $\Phi:=(\lambda_n^{1/2}\varphi_n)_{n\in \dN}$. From  (\ref{f_s}) we derive
$$
f_s(a) =\sum_{n=1}^\infty\, \langle \rho(a) \lambda_n^{1/2}\varphi_n,\lambda_n^{1/2}\varphi_n\rangle =\langle \rho_\Phi (a)\Phi,\Phi \rangle=:f_\Phi(a),~~a\in A ,
$$
that is, $f_s$ is equal to the vector functional $f_\Phi$ in the representaton $\rho_\Phi$. Therefore, by the uniqueness of the GNS representation, $\pi_{f_s}$ is unitarily equivalent to $\rho_\Phi$. 
Likewise, the GNS representation $\pi_{f_t}$ is unitarily equivalent to the corresponding $*$-representation $\rho_\Psi$, where 
$t\varphi=\sum_n \langle \varphi ,\psi_n\rangle\mu_n\psi_n$
 is  a corresponding representation  of\, the operator $t$\, and\, $\Psi:=(\mu_n^{1/2}\psi_n)_{n\in \dN}$. Clearly, since\, 
$\rho_\Phi \subseteq \rho_\infty$\, and\, $\rho_\Phi \subseteq \rho_\infty$,\, we have $\Phi \in S(\rho_\infty, f_s)$ and 
$\Psi \in S(\rho_\infty, f_t)$. 

Let $\cM_\infty$ denote the von Neumann algebra $(\rho_\infty(A)^\prime_{ss})^\prime$. Then, by Theorem \ref{sumofgnspifg1}, we have 
\begin{align}\label{papm}P_A (f_s,f_t)\equiv P_A(f_\Phi,f_\Psi)=P_{\cM_\infty}( F_\Phi, F_\Psi).
\end{align}

Let $x\in {\bf B}(\Hh_\infty)$. We write $x$ as a matrix\, $(x_{jk})_{j,k\in \dN}$\, with entries\, $x_{jk}\in {\bf B}(\Hh(\rho)))$. Clearly, $x$ belongs to in $\rho_\infty(\cA)^\prime_{ss}$\, if and only if each entry $x_{jk}$ is in\, $\rho(A)^\prime_{ss}$. Further, it is easily verified that $x$ is in\,  $(\rho_\infty(\cA)^\prime_{ss})^\prime$\, if and only if there is a (uniquely determined) operator $x_0 \in(\rho(A)^\prime_{ss})^\prime$ such that $x_{jk}=\delta_{jk} x_0$ for all $j,k\in \dN$. The map $\pi(x_0):=x$\, defines a $*$-isomorphism of  von Neumann algebras $\cM:=(\rho(A)^\prime_{ss})^\prime$ and $\cM_\infty=(\rho_\infty(\cA)^\prime_{ss})^\prime$, that is, $\pi$ is a $\ast$-representation of $\cM$. 

As above, we let $F_s$ and $F_t$ denote the normal functionals on  $\cM$ defined by $F_s(x):={\rm Tr}\, xs$ and $F_t(x):={\rm Tr}\, xt$, $x\in \cM$. Repeating the preceding reasoning with $\rho$ and $A$ replaced by $\pi$ and $\cM$, respectively, we obtain $F_s(\cdot ) = \langle \pi_0(\cdot)  \Phi,\Phi \rangle\equiv F_\Phi(\cdot)$ and $F_t=F_\Psi$. Hence $P_\cM(F_s,F_t)=P_{\cM_\infty}(F_\Phi,F_\Psi)$, so that 
\begin{align}\label{pafsftm}
P_A (f_s,f_t) =P_\cM(F_s,F_t)
\end{align} by (\ref{papm}). It is proved in \cite[Corollary 1]{alberti03} (see also \cite{uhlmann76}) that 
$$P_\cM(F_s,F_t)= ({\rm tr}_\cM\,|\hat{t}^{1/2}\, \hat{s}^{1/2}|)^2.$$
Combined with (\ref{pafsftm}) this yields  (\ref{fsfttrans})  and completes the proof. 
\end{proof}
Let us remain the assumptions and the  notations of Theorem \ref{tracegen}. In general, $P_A(f_s,f_t)$ is  different from $({\rm Tr}\, ({s}^{1/2}\, {t}\, {s}^{1/2})^{1/2}\, )^2$ a simple examples show. Howover, if in addition $\rho$ is {\it irreducible} (that is, if $\rho(A)^\prime_{ss}=\dC\cdot I$), then $s=\hat{s}$ and $t=\hat{t}$ as noted above and therefore by (\ref{fsfttrans0}) we have
\begin{align}\label{trairr}
P_A(f_s,f_t)= (\,{\rm Tr}\, (s^{1/2}\, t\, s^{1/2})^{1/2}\, )^2.
\end{align}

\smallskip

We now apply  the preceding theorem to an interesting example. 
\begin{thex}\label{schroding} ({\it Schr\"odinger representation of the Weyl algebra})\\
Let $A$ be the Weyl algebra, that is, $A$ is the unital $*$-algebra generated by two hermitian generators $p$ and $q$ satisfying $$pq-qp=-\ii 1 ,$$ and let $\rho$ be the Schr\"odinger representation of $A$, that is,  
\begin{align}\label{schrorep}(\rho(q)\varphi )(x)=x\varphi (x), \quad (\rho(p)\varphi)(x)=-\ii \frac{d}{dx}\varphi(x),\quad \varphi \in\cD(\rho):=\cS(\dR),
\end{align}
 on $L^2(\dR)$. Since $\rho$ is irreducible,  $\rho_\infty(A)^\prime_{ss}=\dC \cdot I$. Hence\, $\cM={\bf B}(\Hh(\rho))$\, and\, ${\rm tr}_\cM={\rm Tr}$. Therefore, if $s,t\in {\bf B}(\pi(A))_+$ and the GNS representations  $\pi_{f_s}$ and $\pi_{f_t}$ are essentially self-adjoint, it follows from  Theorem \ref{tracegen} and  formula (\ref{trairr})  that
  \begin{align}\label{fsfttrans}
P_A(f_s,f_t)= ({\rm Tr}\,|t^{1/2}s^{1/2}|)^2= (\,{\rm Tr}\, (s^{1/2}ts^{1/2})^{1/2}\, )^2.
\end{align}

Let us specialize this to the rank one  case, that is, let $s=\varphi\otimes \varphi$ and $t=\psi\otimes \psi$ with $\varphi,\psi \in \cD(\rho)$, so that $f_s(a)=\langle \rho(a)\varphi,\varphi\rangle$ and $f_t(a)=\langle \rho(a)\psi,\psi\rangle$ for $a\in A$. Then formula (\ref{fsfttrans}) yields  
\begin{align}\label{fsftvector}
P_A(f_s,f_t)= |\langle \varphi,\psi\rangle |^2.
\end{align}
Recall that (\ref{fsftvector})  holds under the assumption that the GNS representations  $\pi_{f_s}$ and $\pi_{f_t}$ are essentially self-adjoint.
We shall see in section\, \ref{vectorweyl} below that  (\ref{fsftvector}) is no longer true if the latter assumption is omitted.
\end{thex}
Now we turn to the second main application.
\begin{tht}\label{commtrans} Let $X$ be a locally compact topological Hausdorff space and let $A$ be a  $*$-subalgebra of\, $C(X)$\,  which contains the constant function $1$ and separates the points of $X$.  Let $\mu$ be a positive regular Borel measure on $X$ such that $A\subseteq L^1(X,\mu)$ and let $\eta, \xi\in L^\infty(X,\mu)$ be nonnegative functions. Define positive linear functionals $f_\eta$ and $f_\xi$ on $A$ by
\begin{align}\label{fetaxi} f_\eta(a)=\int_X a(x)\eta(x)\, d\mu(x),\quad  f_\xi(a)=\int_X a(x)\xi(x)\, d\mu(x)\, \, {\rm for}\,\,  a\in  A.
\end{align}
Suppose that the GNS representations  $\pi_{f_\eta}$ and $\pi_{f_\xi}$ are essentially self-adjoint. Then 
\begin{align}\label{commutativecase}
P_\cA(f_\eta,f_\xi)= \bigg(\,\int_X\, \eta(x)^{1/2}\xi(x)^{1/2}\, d\mu(x) \bigg)^2. 
\end{align}
\end{tht}
\begin{proof}
We define a closed $*$-representation $\pi$ of the $*$-algebra $A$ on $L^2(X,\mu)$ by 
$\pi(a)\varphi=a\cdot \varphi$ for $a\in A$ and $\varphi$ in the domain $$\cD(\pi):=\{\varphi \in L^2(X,\mu):a\cdot \varphi\in L^2(X,\mu) ~{\rm for}~ a\in A\}.$$

 First we prove that $\pi(A)^\prime_{ss}=L^\infty (X,\mu)$, where the functions of\, $L^\infty (X,\mu)$ act as multiplication operators on\, $L^2 (X,\mu).$ Let $\gA$ denote the $*$-subalgebra of $L^\infty (X,\mu)$ generated by the functions $(a\pm \ii)^{-1}$, where $a=a^+\in A$. Obviously, $L^\infty (X,\mu)\subseteq \pi(A)^\prime_{ss}$. 
 Conversely, let $x\in \pi(A)^\prime_{ss}$. It is straightforward to show that for any $a=a^+\in A$ the operator\, $\ov{\pi(a)}$\, is self-adjoint and hence equal to the (self-adjoint) multiplication operator by the function $a$. By definition $x$ commutes with\, $\ov{\pi(a)}$\,, hence with\, $(\ov{\pi(a)}\pm \ii I)^{-1}=(a\pm \ii)^{-1}$, and therefore with the whole algebra $\gA$. The $*$-algebra $A$ separates the points of $X$, so does the $*$-algebra $\gA$. Therefore, from the Stone--Weierstrass theorem \cite[Corollary 8.2]{conway90}, applied to the one point compactification of $X$, it follows that $\gA$ is norm dense in $C_0(X)$. Hence $x$ commutes with $C_0(X)$ and so with its closure \, $L^\infty (X,\mu)$\, in the weak operator topology. Thus, $x\in L^\infty (X,\mu)^\prime$. Since $ L^\infty (X,\mu)^\prime= L^\infty (X,\mu)$, we have shown that $\pi(A)^\prime_{ss}=L^\infty (X,\mu)$. Therefore, $\cM:=(\pi(A)^\prime_{ss})^\prime=L^\infty (X,\mu)$.
 
 Let $F_\eta$ and $F_\xi$ denote the positive linear functionals on $\cM$ defined by (\ref{fetaxi}) with $A$ replaced by $\cM$. 
  For $\cM=L^\infty (X,\mu)$ it is well-known (see e.g. formula (14) in \cite{alberti83}) that $P_\cM(F_\eta,F_\xi)= (\,\int_X\, \eta(x)^{1/2}\xi(x)^{1/2}\, d\mu(x) \,)^2. $ Since $P_A(f_\eta,f_\xi)=P_\cM(F_\eta,F_\xi)$ by Theorem \ref{sumofgnspifg1}, we obtain (\ref{commutativecase}).
  \end{proof}

In the following two examples we reconsider the one dimensional Hamburger moment problem (see Example \ref{hambmp}) and we specialize the preceding theorem to the case where $X=\dR$ and $A=\dC[x]$.
\begin{thex}\label{detmpform} {\it  Determinate Hamburger moment problems}\\
Let $\mu_\eta$ and $\mu_\xi$ be the positive Borel measures on $\dR$ defined by  $d\mu_\eta=\eta d\mu$ and $d\mu_\xi =\xi d\mu$. Since $\dC[x]\in L^1(\dR,\mu)$ and $\eta, \xi\in L^\infty(\dR,\mu)$, we have $\mu_\eta,\mu_\xi\in M(\dR)$.  If both measures $\mu_\eta$ and $\mu_\xi$ are determinate, then the GNS representations $\pi_{f_{\mu_\eta}}$ and $\pi_{f_{\mu_\xi}}$ are essentially self-adjoint (as shown in Example \ref{hambmp}) and hence formula (\ref{commutativecase}) holds by Theorem \ref{commtrans}.
\end{thex}

\begin{thex}\label{indetr} {\it {Indeterminate Hamburger  moment problems}}\\
Suppose  $\nu\in M(\dR)$ is an indeterminate measure such that $\nu(\dR)=1$. 

Let $V_\nu$ denote  the set of all positive Borel measures $\mu\in M(\dR)$ which have the same moments as $\nu$,  that is, $\int x^n d\nu(x)=\int x^n d\mu(x)$ for all $n\in \dN_0$. Since $\nu$ is indeterminate and $V_\nu$ is  convex and weakly compact, there exists a measure $\mu\in V_\nu$ which is not an extreme point of $V_\nu$, that is, there are measures $\mu_1,\mu_2\in V_\nu$, $\mu_j\neq \mu$ for $j=1,2$, such that $\mu=\frac{1}{2}(\mu_1+\mu_2)$. Since $\mu_j(M)\leq 2\mu(M)$ for all measurable sets $M$ and $\mu_1+\mu_2=2\mu$, there exists  functions $\eta, \xi \in L^\infty(\dR,\nu)$ satisfying 
\begin{align}\label{propetaxi} 
\eta(x)+\xi(x)=2 ,~~  \|\xi\|_\infty \leq 2,~~ \|\eta\|_\infty \leq 2,~~  d\mu_1=\eta d\mu,~~ d\mu_2=\xi d\mu.
\end{align}
 
Define $f(p)=\int p(x)d\mu(x)$ for $p\in \dC[x]$. Since $\mu_1,\mu_2,\mu\in V_\nu$, the functionals $f_\eta$ and $f_\xi$ defined by (\ref{fetaxi}) are equal to $f$. Therefore, since $f(1)=\mu(\dR)=\nu(\dR)=1$, we have $P_A(f_\eta,f_\xi)=P_A(f,f)=1$.  
 
Put $J:=(\,\int_X\, \eta(x)^{1/2}\xi(x)^{1/2}\, d\mu(x) )^2.$ From (\ref{propetaxi}) we obtain $\eta(x)\xi(x)=\eta(x)(2-\eta(x))\leq 1$ and hence  $J\leq 1$, since $\mu(\dR)=1$. If $J$ would be equal to $1$, then $\eta(x)(2-\eta(x))=1$ $\mu$-a.e. on $\dR$  which  implies that $\eta(x)=1$\, $\mu$-a.e. on $\dR$ by (\ref{propetaxi}). But then $\mu_1=\mu_2=\mu$ which contradicts the choice of measures $\mu_1$ and $\mu_2$. Thus we have proved that $J\neq 1=P_A(f_\eta,f_\xi)$, that is, formula (\ref{commutativecase}) does not hold in this case.
\end{thex}
The classical moment problem leads to a number of open problems concerning  transition probabilities. We will state  three of them. 

Let  $M(\dR^d)$, $d\in \dN$, denote the set of positive Borel measures $\mu$ on $\dR^d$ such that all polynomials $p(x_1,\dots,x_d)\in \dC[x_1,\dots,x_d]$ are $\mu$-integrable. For $\mu \in M(\dR^d)$ we define a positive linear functional $g_\mu$ on the $*$-algebra $A:=\dC[x_1,\dots,x_d]$ by 
$$g_\mu(p)=\int p\, d\mu,\quad p \in \dC[x_1,\dots,x_d].$$
Then the main problem is the following:\\
{\it Problem 1: Given $\mu,\nu \in M(\dR^d)$, what is\, $P_A(g_\mu,g_\nu)$ ?}

This seems to be a difficult  problem and it is hard to  expect   a sufficiently complete answer.
For $d=1$ Example \ref{detmpform} contains  some answer under the assumption that both  measures $\mu_\eta$ and $\mu_\xi$ are determinate. This suggests the following  questions:\\
{\it Problem 2: What about the case  when the measures $\mu_\eta$ and/or $\mu_\xi$ in Example \ref{detmpform} are not determinate?}\\
{\it Problem 3: Is  formula (\ref{commutativecase}) 
still valid in the multi-dimensional case $d>1$ if  $\mu_\eta$ and $\mu_\xi$ are determinate ?}

It can be shown that the answer to problem 3 is affirmative if all multiplication operators $\pi_\mu(x_j)$, $j=1,\dots,d,$ are essentially self-adjoint. The latter assumption is sufficient, but not neccessary for $\mu$ being determinate \cite{ps}. 
 In the multi-dimensional case determinacy turns out to be much more difficult than in the  one-dimensional case, see e.g. \cite{ps}. 

\section{Vector Functionals of the Schr\"odinger Representation}\label{vectorweyl}
The crucial assumption for the  results in preceding sections was the essential self-adjointness of GNS representations $\pi_f$ and $\pi_g$. In this section we consider the simplest situation where $\pi_f$ and $\pi_g$ are not essentially self-adjoint.

In this section  $A$ denotes the Weyl algebra (see Example \ref{schroding}) and   $\pi$ is the Schr\"odinger representation of $A$ given by (\ref{schrorep}). For $\eta \in \cD(\pi)=\cS(\dR)$ let $f_\eta$ denote the positive linear functional $f_\eta$  on $A$ given by $$f_\eta(x)=\langle \pi(x)\eta,\eta\rangle,\quad  x\in A.$$ Consider the following condition on the function $\eta$:

\smallskip
$(*)$ {\it There are finitely many mutually disjoint open intervals $J_l(\eta)=(\alpha_l,\beta_l)$, $l=1,\dots,r$,  such that $\eta(t)\neq 0$ for  $t\in J(\eta):=\cup_l\, J_l(\eta) $ and $\eta^{(n)}(t)=0$ for  $t\in \dR/ J(\eta)$ and all $n\in \dN_0$.}

\smallskip
The main result of this section is the following theorem.
\begin{tht}\label{weylstates}
Suppose that $\varphi$ and $\psi$ are functions of $C_0^\infty(\dR)$ satisfying condition $(*)$. Then 
\begin{align}\label{vectorstates}
P_A(f_\varphi,f_\psi)= \bigg(\sum_{k,l} \bigg|\int_{\cJ_k(\varphi)\cap\cJ_l(\psi)}\, \varphi(x) \ov{\psi(x)}\, dx \bigg|\bigg)^2.
\end{align}
(If ~  $\cJ_k(\varphi)\cap\cJ_l(\psi)$~ is empty, the corresponding integral is set zero.) 
\end{tht}

Before we turn to the proof of the theorem let us discuss formula (\ref{vectorstates}) in two simple cases. 

$\bullet$ If both sets $\cJ(\varphi)$ and $\cJ(\psi)$ consist of a  single interval, then  
$$P_A(f_\varphi,f_\psi)= \bigg|\int_\dR\, \varphi(x) \ov{\psi(x)}\, dx \bigg|^2=|\langle \varphi,\psi\rangle|^2,$$ 
that is, in this case formula (\ref{fsftvector}) holds. 

$\bullet$ Let $\varphi, \psi\in C_0^\infty(\dR) $ be such that $\cJ(\varphi)=\cJ(\psi)$,  $\cJ_k(\varphi)=\cJ_k(\psi)$  and $\varphi(x)=\epsilon_k\psi(x)$ on $\cJ_k(\varphi)$ for $k=1,\cdots,r$, where $\epsilon_k\in \{1,-1\}$. Then formula (\ref{vectorstates}) yields\, $P_A(f_\varphi,f_\psi)=\|\varphi\|^4$.\, It is easy to choose  $\varphi\neq 0$ and the numbers $\epsilon_k$ such that $\langle \varphi,\psi\rangle=0$, so  formula (\ref{fsftvector}) does not hold in this case.

\medskip
The proof of Theorem \ref{weylstates} requires a number of technical preparations. The  first  aim is to desribe the closure $\overline{\pi}_{f_\eta}$ of the GNS representation $\pi_{f_\eta}$ for a function $ \eta\in C_0^\infty(\dR)$ satisfying  condition $(*)$. 
 
Let $\rho_\eta$ denote the restriction of $\pi$ to the dense domain 
\begin{align*}
\cD(\rho_\eta)=
\{ \xi\in \bigoplus_{l=1}^r C^\infty((\alpha_l,\beta_l)): \xi^{(k)}(\alpha_l)=\xi^{(k)}(\beta_l)=0,\, 
k\in \dN_0,\, l=1,\dots,r \}
\end{align*}
in the Hilbert space $L^2(\cJ(\eta))$. The following lemma  says that\, $\rho_\eta$\, is unitarily equivalent to $\overline{\pi}_{f_\eta}$.

\begin{thl}\label{gnseta} There is a unitary operator $U$ of\, $\Hh(\pi_{f_\eta})$\, onto $L^2(\cJ(\eta))$ given by 
$U(\pi_{f_\eta}(a)\eta)= \rho_\eta(a)\eta$,  $a\in A$, such that\,  $\rho_\eta=U\ov{\pi}_{f_\eta} U^\ast$.
\end{thl} 
\begin{proof}
From the properties of GNS representations it follows easily that the unitary  operator $U$ defined by $U(\pi_{f_\eta}(a)\eta)= \rho_\eta(a)\eta$,  $a\in A$, provides unitary equivalences $ \tau_\eta=U\pi_{f_\eta} U^\ast$ and $\ov{\tau}_\eta=U\ov{\pi}_{f_\eta} U^\ast$, where $\tau_\eta$ denotes the restriction of $\pi$ to  $\cD(\rho_\eta)= \pi(A)\eta$. Clearly, $\tau_\eta\subseteq \rho_\eta$ and  hence $\ov{\tau}_\eta\subseteq \rho_\eta$, since $\rho_\eta$ is obviously closed. To  prove  the statement it therefore suffices to show that $\rho_\eta$ is the closure of $\tau_\eta$, that is, $\pi(A)\eta$ is dense in $\cD(\rho_\eta)$ in the graph topology of\,  $\rho_\eta(A)$. For this  the auxiliary Lemmas \ref{denselemma1} and \ref{denselemma2} proved below are essentially used.

Each element $a\in A$ is of a finite sum of terms $f(q)p^n$, where $n\in \dN_0$ and $f\in \dC[q]$. Since $ \eta\in C_0^\infty(\dR)$, the set\, $\cJ(\eta)$ and hence the operators $\pi_0(f(q))$ are bounded. Therefore, the graph topology $\gt_{\rho_\eta(A)}$ is generated by the seminorms $\|\rho_\eta(p)^n \cdot\|$, $n\in \dN_0$, on $ \cD(\rho_\eta)$. Let $\psi\in \cD(\rho_\eta)$.

First assume that $\psi$ vanishes in some
neighbourhoods of the end points $\alpha_l, \beta_l$. Then, by Lemma \ref{denselemma2}, for any $m\in\dN$ there is sequence\, $(f_n)_{n\in \dN}$\, of polynomials such that $$\lim_n~ \rho_\eta((\ii p)^k)(\rho_\eta( f(q))\eta-\psi)=\lim_n~((f_n\eta)^{(k)} -\psi^{(k)})=0 $$ in $L^2(\cJ(\eta))$  for $k=0,\dots,m$. 
This shows that $\psi$ is in the closure of  $\rho_\eta(A)\eta$ with respect the graph topology of $\rho_\eta(A)$.

The case of a general function $\psi$ is reduced to the preceding case as follows. Suppose that $\varepsilon>0$ and~ $2\varepsilon <\min_l\, |\beta_l-\alpha_l|$. We define
\begin{align*}
\psi_\varepsilon(x)=\psi(x-\varepsilon+2\varepsilon(x-\alpha_l-\varepsilon)(\beta_l-\alpha_l-2\varepsilon)^{-1})\quad{\rm for}\quad x\in (\alpha_l,\beta_l)
\end{align*}
and $l=1,\dots,r$ and $\psi_\varepsilon(x)=0$ otherwise. Then $\psi_\varepsilon$ vanishes in some neighbourhoods of the end points\, $\alpha_l, \beta_l$, so  $\psi_\varepsilon$ is in the closure of $\rho_\eta(A)\eta$ as shown in  the preceding paragraph. Using the  dominated Lebesgue convergence theorem it follows that 
$$\lim_{\varepsilon\to +0}\rho_\eta((\ii p)^k)(\psi_\varepsilon-\psi)= \lim_{\varepsilon\to +0} (\psi_\varepsilon^{(k)}-\psi ^{(k)})=0$$ in\, $L^2(\cJ(\eta))$  for $k\in \dN_0$. Therefore, since  $\psi_\varepsilon$ is in the closure of $\rho_\eta(A)\eta$, so is $\psi$.
\end{proof} 
\begin{thl}\label{denselemma1}
Suppose that\, $g\in C^{(k)}([\alpha,\beta])$,\, where $\alpha,\beta \in \dR$ and $k\in \dN$. Then there exists a sequence $(f_n)_{n\in \dN}$ of polynomials such that\, $f^{(j)}_n(x)\Longrightarrow g^{(j)}(x)$\, uniformly on\, $[\alpha,\beta]$\, for $j=0,\dots,k$  as\, $n\to \infty$. 
\end{thl}
\begin{proof}
By the Weierstrass theorem there is a sequence $(h_n)_{n\in \dN}$ of polynomials such that $h_n(x) \Longrightarrow g^{(k)}(x)$ uniformly on $[\alpha,\beta]$. Fix $\gamma\in [\alpha,\beta]$ and set\, $h_{n,k}:=h_n$. Then $$h_{n,k-1}(x):=g^{(k)}(\gamma)+ \int^x_\gamma h_{n,k}(s)\, ds \Longrightarrow  g^{(k-1)}(x)=g^{(k)}(\gamma)+\int^x_\gamma g^{(k)}(s)\, ds. $$
 Clearly,  $(h_{n,k-1})_{n\in \dN}$ is sequence of polynomials and we have $h_{n,k-1}^\prime(x)=h_{n,k}(x)$ on $[\alpha,\beta]$. Proceeding by induction we obtain sequences $(h_{n,k-j})_{n\in \dN}$, $j=0,\cdots,k$,  of polynomials such that\, $h_{n,k-j}(x)\Longrightarrow g^{(k-j)}(x)$\,  and\, $h_{n,k-j}^\prime(x)=h_{n,k+1-j}(x)$\, on $[\alpha,\beta]$. Setting\, $f_n:=h_{n,0}$\, the sequence\, $(f_n)_{n\in \dN}$\, has the desired properties.
\end{proof}

\begin{thl}\label{denselemma2}
Suppose that $\eta\in C_0^\infty (\dR)$ satisfies condition $(*)$. Let  $\psi\in \bigoplus_{l=1}^r C^{(m)}_0((\alpha_l,\beta_l))$, where $m\in \dN$. 
Then there exists a sequence $(f_n)_{n\in \dN}$ of polynomials such that~ $\lim\nolimits_{n\to \infty}\,(f_n\eta)^{(k)}= \psi^{(k)}$\, in $L^2(\cJ(\eta))$  for $k=0,\dots,m$. 
\end{thl}
\begin{proof}
By the assumption  $\psi$ vanishes in some neighbourhoods of the end points $\alpha_l$ and $\beta_l$. 
Set $\psi(x)=0$ on $\dR/ \cJ(\eta)$. Then, $\psi\eta^{-1}$ becomes a function\, of $ C^{(m)}([\alpha,\beta])$, where $\alpha:=\min_l \, \alpha_l$ and $\beta:=\max_l\, \beta_l$. Therefore, by Lemma \ref{denselemma1}, there exists  a sequence $(f_n)_{n\in \dN}$ of polynomials such that\, $f^{(j)}_n(x)\Longrightarrow (\psi \eta^{-1})^{(j)}(x) $\, for $j=0,\dots, m$\,  uniformly on\,  $[\alpha,\beta]$. Then
\begin{align*}
(f_n\eta)^{(k)}=\sum_{j=0}^k\, {k \choose j} f_n^{(j)}\eta^{(k-j)}\Longrightarrow \sum_{j=0}^k\, { k \choose j} 
(\psi \eta^{-1})^{(j)}\eta^{(k-j)} =\psi^{(k)}
\end{align*}
as\, $n\to \infty$\, uniformly on\, $[\alpha,\beta]$\,  and hence in $L^2(\cJ(\eta))$.
\end{proof}

Now we are able to give the\\
{\it Proof of Thorem \ref{weylstates}:}
Let us abbreviate $\pi_\varphi=\ov{\pi}_{f_\varphi}$ and $\pi_\psi=\ov{\pi}_{f_\psi}$. By Lemma \ref{gnseta} the closure $\pi_\psi=\ov{\pi}_{f_\psi}$ of the GNS representation $\pi_{f_\psi}$ is unitarily equivalent to the representation $\rho_\psi$. For notational simplicy we shall identify the representations $\pi_\psi$ and $\rho_\psi$ via the unitary $U$ defined in Lemma \ref{gnseta}. Using this description of $\pi_\psi\cong \rho_\psi$ it is straightforward to check that the domain $\cD((\pi_\psi)^*)$ consists of all functions $g\in C^\infty(\cJ(\psi))$ such that their restrictions to $\cJ_l(\psi)$ extend to functions of $C^\infty(\,\ov{\cJ_l(\psi)}\, )$ and $g(t)=0$ on $ \dR/\, \ov{\cJ(\psi)}$. Further, we have $(\pi_\psi)^*(f(q))g =f \cdot g $  and $(\pi_\psi)^*(p)g =-\ii g^\prime $ for $g \in \cD((\pi_\psi)^*)$ and  $f\in \dC[q]$.

Suppose that $T\in I(\pi_\varphi, (\pi_\psi)^*)$ and $\|T\|\leq 1$. Set $\xi:=T\varphi$. 
By the intertwining property of $T$,  for each polynomial $f$ we have
\begin{align}\label{intfq}
T(f\cdot\varphi)=T\pi_\varphi(f(q))\varphi =(\pi_\psi)^*(f(q))T\varphi=(\pi_\psi)^*(f(q))\xi=f\cdot \xi .
\end{align} 
Therefore, since $\|T\|\leq 1$, we  obtain
\begin{align}\label{tless}
 \int_\alpha^\beta |f(x)|^2|\xi(x)|^2dx= \int_\alpha^\beta |T(f\cdot \varphi)(x)|^2dx \leq \int_\alpha^\beta |f(x)|^2|\varphi(x)|^2 dx
\end{align}
for all polynomials $f$ and hence for all functions $f\in C[\alpha,\beta]$ by the Weierstrass theorem. Hence\, (\ref{tless}) implies that
\begin{align}\label{xivarphi}
|\xi(x)|\leq |\varphi(x)|\quad {\rm  on}\quad [\alpha,\beta].
\end{align} 
Therefore, $\xi(x)=0$ if\, $x\in \dR/\, \cJ(\varphi)$. Clearly, $\xi(x)=0$ if $x\in \dR/\, \ov{\cJ(\psi)}$, since $\xi\in \cD((\pi_\psi)^*)$. Since $\varphi$ satisfies condition $(*)$,  the set $\{f\cdot \varphi:f\in \dC[x]\}$ is dense in $L^2(\cJ(\varphi))=\Hh(\pi_\varphi)$.  Therefore, it follows from  (\ref{intfq}) that  $T$ is equal to the multiplication operator by the bounded function $\xi \varphi^{-1}$. (Note that $\xi \varphi^{-1}$ is bounded by (\ref{xivarphi}).) In particular, we obtain
 $$\varphi^\prime \cdot\xi \varphi^{-1} =T\varphi^\prime=T\pi_\varphi(\ii p)\varphi =(\pi_\psi)^*(\ii p)T\varphi= (\pi_\psi)^*(\ii p)\xi= \xi^\prime$$
Thus,\, $\varphi^\prime(x)\xi(x)=\varphi(x) \xi^\prime(x)$\, which in turn implies that\, $(\frac{\xi}{\varphi})^\prime(x)=0$\, for all $x\in \cJ(\varphi)\cap\cJ(\psi)$. Hence $\frac{\xi}{\varphi}$ is constant, say $\xi(x)=\lambda \varphi(x)$ for some constant $\lambda \in \dC$  on each connected component  of $\cJ(\varphi)\cap\cJ(\psi)$. By (\ref{xivarphi}),\, $|\lambda|\leq 1$. The connected components of the open set $\cJ(\varphi)\cap\cJ(\psi)$ are precisely the intervals $\cJ_l(\varphi)\cap\cJ_k(\psi)$  provided the latter  is not empty.

Conversely, suppose that for all indices $l,k$ such that\, $\cJ_l(\varphi)\cap \cJ_k(\psi)\neq \emptyset$\, a complex number $\lambda_{k,l}$, where\, $|\lambda_{k,l}|\leq 1$, is given. Set 
$\xi(x)=\lambda_{kl} \varphi(x)$ for   $x\in \cJ_l(\varphi)\cap\cJ_k(\psi)$  and  $\xi(x)=0$ otherwise. From the  description of the domain\, $\cD((\pi_\psi)^*)$\, given in the first paragraph of this proof it follows that $\xi\in \cD((\pi_\psi)^*)$. Define  $T(\pi_\varphi(a)\varphi):=(\pi_\psi)^*(a)\xi$, $a\in A$. It is easily checked that  $T$ extends by continuity to an operator $T$ of $\Hh(\pi_\varphi)=L^2(\cJ(\varphi))$ into $\Hh((\pi_\psi)^*)=L^2(\cJ(\psi))$ such that  $T\in I(\pi_\varphi,( \pi_\psi)^*)$\, and $\|T\|\leq 1$. Since $T\varphi=\xi$, we have 
$$\langle T\varphi,\psi\rangle=\sum\nolimits_{k,l}~ \lambda_{k,l} \int_{\cJ_k(\varphi)\cap\cJ_l(\psi)}\, \varphi(x) \ov{\psi(x)}\, dx.$$
Therefore, the supremum of  expressions\,  $|\langle T\varphi,\psi\rangle|$ is obtained if we choose  $\lambda_{k,l}$ such that the number 
$\lambda_{k,l} \int_{\cJ_k(\varphi)\cap\cJ_l(\psi)}\, \varphi \ov{\psi}\, dx$\, is  equal to its modulus\,  $|\int_{\cJ_k(\varphi)\cap\cJ_k(\psi)}\, \varphi \ov{\psi}\, dx|$. This implies  formula (\ref{vectorstates}).\hfill$\Box$

\bigskip

\noindent{\bf Acknowledgements.} The author would like to thank  P.M. Alberti for many fruitful discussions on transition probabilities.

\bibliographystyle{amsalpha}

\end{document}